\documentclass{amsart}

\usepackage{verbatim}

\usepackage{amsopn}
\usepackage{amsfonts}
\usepackage{algorithmic}
\usepackage{enumerate}
\usepackage{graphicx,epstopdf}
\usepackage[caption=false]{subfig}
\usepackage{multirow,array}
\usepackage{color}
\usepackage{amsmath}
\usepackage{hyperref}

\theoremstyle{plain}
\newtheorem{theorem}{Theorem}[section]
\newtheorem{lemma}[theorem]{Lemma}
\newtheorem{proposition}[theorem]{Proposition}
\newtheorem{corollary}[theorem]{Corollary}
\newtheorem{algorithm}[theorem]{Algorithm}
\numberwithin{equation}{section}
\theoremstyle{theorem}

\newtheorem{remark}{Remark}[section]
\newtheorem{assumption}{Assumption}[section]

\def\Fto{{}_2F_1}

\def\cM{{ M}}

\def\dH{\dot{H}}
\def\calC{\mathcal C}

\def\Forall{\qquad \hbox{for all }}

\def\beq#1\eeq{\begin{equation} #1 \end{equation}}
\def\bal#1\eal{\begin{aligned} #1 \end{aligned}}

\def\RR{{\mathbb R}}

\def\Lnt{L}
\def\Kind{{K}}
\def\hu{\hat{u}}

\def\sumj{\sum_{j=1}^\infty}
\chardef\atsign='100

\def\calA{{\mathcal A}}
\def\calB{{\mathcal B}}
\def\calI{{\mathcal I}}
\def\calT{{\mathcal T}}
\def\ff{ v}

\newcommand{\INDSTATE}[1][1]{\STATE\hspace{#1\algorithmicindent}}

\title[Approximation of Fractional Powers of Elliptic Operators]{Numerical Approximation of Fractional Powers of Elliptic Operators} 

\author{Beiping Duan
	\and Raytcho D.~Lazarov
	\and Joseph E.~Pasciak
}

\address{Beiping Duan, School of Mathematics and Statistics, Central South University, 410083 Changsha, P.R. China and
Department of Mathematics, Texas A\&M University, College Station, TX 77843-3368, USA}
\email{duanbeiping@hotmail.com}

\address{Raytcho Lazarov, Department of Mathematics, Texas A\&M University, College Station,
TX~77843-3368 and Institute of Mathematics and Informatics, Bulgarian Academy of Sciences, acad. G. Bonchev str., blok 8,
1113 Sofia, Bulgaria.}
\email{lazarov\atsign math.tamu.edu}

\address{Joseph E. Pasciak, Department of Mathematics, Texas A\&M University, College Station,
TX~77843-3368.}
\email{pasciak\atsign math.tamu.edu}

\date{Started May 2017, today is \today}

\begin{document}


\begin{abstract} 
	In this paper, we develop and study algorithms  for approximately solving the linear algebraic systems: 
	$\mathcal{A}_h^\alpha u_h = f_h$, $ 0< \alpha <1$, for $u_h, f_h \in
	V_h$ with  $V_h$ a finite element approximation space.
	Such problems arise in finite element or finite difference
	approximations of the  
	problem $ \calA^\alpha u=f$ with $\calA$, for example, coming from   
	a second order elliptic operator  with homogeneous boundary
	conditions. 
	The algorithms are motivated by  the method of Vabishchevich
	\cite{Vabishchevich2015JCP} that relates the algebraic problem to a
	solution of a time-dependent initial value problem on the interval
	$[0,1]$.   
	Here we develop and study two time stepping schemes 
	based on diagonal Pad\'e approximation to $(1+x)^{-\alpha}$.
	The first one uses geometrically graded meshes in order to compensate
	for the singular behavior of the solution for $t$ close to $0$. 
	The second algorithm uses uniform time stepping  
	but requires smoothness of the data $f_h$ in discrete norms. 
	For both methods,  we estimate the error in terms of the number of time
	steps, with the regularity of $f_h$ playing a major role for the second method. 
	Finally, we present numerical experiments for $\calA_h$ coming from the
	finite element approximations of  second order elliptic boundary value problems in 
	one and two spatial dimensions. 
\end{abstract}





\maketitle


\section{Introduction}
\subsection{Motivation and problem formulation}
%
Nonlocal operators arise in a wide variety  of mathematical
models such as modes of long-range interaction in elastic deformations
\cite{silling2000reformulation}, nonlocal electromagnetic fluid flows \cite{mccay1981theory}, 
image processing \cite{gilboa2008nonlocal,pu2010fractional} and many more.
A recent discussion about the properties of such models and their applications to chemistry, 
geosciences, and engineering can be found in \cite{KilbasSrivastavaTrujillo:2006,metzler2014anomalous}.

The nonlocal operators considered in this 
paper involve fractional powers of operators $\calA$ associated with
second order elliptic equations  in bounded domains with homogeneous 
Dirichlet boundary conditions. The fractional power of $\calA$  
is defined through the Dunford-Taylor integral, \cite{kato1961,lunardi}, 
which is equivalent to the definition by the spectrum of $\calA$. For a detailed
discussion about this setting and other possible ways to define fractional powers of
the Laplacian (and more general elliptic operators) we refer to \cite{Bonito2018,kwasnicki2017ten,Karniadakis2018fractional}.  
We focus on issue of solving the corresponding algebraic system 
that arises in approximating such operators by the finite element method, 
e.g. \cite{Bonito2018, BP-frac,wenyu-thesis}.

We begin with the definition of the fractional power of a second order elliptic operator 
in a bounded domain $\Omega \subset \RR^d$, $d=1,2,\ldots$ with a Lipschitz continuous boundary.
On  $V \times V$, with $V=H^1_0(\Omega)$,  we consider the bilinear form:
\beq\label{bi-form}
A(w,\phi)= \int_\Omega  \Big (a(x) \nabla w \cdot\nabla \phi +q(x) w \phi \Big )\, dx, 
\eeq
and assume that the coefficients are such that  the bilinear form is  coercive  and bounded on $V$. 
 Further, we define $\calT:L^2(\Omega) \rightarrow V$, where $\calT v=w \in V$ is the unique solution to
\beq A(w,\phi)=(v,\phi),\qquad \phi\in V.
\label{weak}
\eeq
Here $(\cdot,\cdot)$ denotes the $L^2(\Omega)$-inner product.

Following \cite{kato1961}, we define an unbounded operator $\calA$ with
domain of definition $D(\calA)$ being the image of $\calT$ on
$L^2(\Omega)$ and set $\calA u= \calT^{-1} u$ for $u\in D(\calA)$.
This is well defined as $\calT$ is injective on $L^2(\Omega)$.
Negative fractional powers can be defined by  Dunford-Taylor integrals,
i.e., for $\alpha>0$ and $v\in L^2(\Omega)$,
\[
\calA^{-\alpha}v = \frac 1 {2 \pi i} \int _{\calC} z^{-\alpha}
R_z(\calA)v
\, dz
\] 
where $R_z(\calA)=(\calA-z \calI)^{-1}$ is the resolvent operator 
and $\calC$ is an appropriate contour in the complex plane (see, e.g., \cite{lunardi}).

Equivalently, fractional powers for the above example
can be defined by eigenvector expansions.
As $\calT$ is a
compact, symmetric and positive  definite operator on $L^2(\Omega)$, its eigenpairs
$\{\psi_j,\mu_j\}$, for $j=1,2,\ldots,\infty$, with suitably
normalized eigenvectors, provide an orthonormal basis
for $L^2(\Omega)$.  We also set $\lambda_j=\mu_j^{-1}$.   For $\alpha\ge 0$
and $v\in L^2(\Omega)$, 
\[
\calA^{-\alpha} v= \sum_{j=1}^\infty   \mu_j^\alpha (v,\psi_j) \,\psi_j.
\]
Positive fractional powers of $\calA$ are also given by similar series.
For $\alpha \ge 0$, define
\[
D(\calA^\alpha):= \{ v \in L^2\  : \ \sumj  \lambda_j^{2\alpha}
|(v,\psi_j)|^2<\infty\}
\]
and
\[
\calA^\alpha  v:= \sumj \lambda_j^\alpha (v,\psi_j) \,\psi_j \   \mbox{ for } \   v \in D(\calA^\alpha).
\]

We consider the  fractional order elliptic equation:  Find $u\in D(\calA^\alpha)$ satisfying
\beq\label{fracprob}
\calA^\alpha u = f \ \ \mbox{for } \ \ f\in L^2(\Omega)
\eeq
and note that its solution is given by
\[
u= \calT^\alpha  f:=\sumj \mu_j^\alpha (f,\psi_j)\, \psi_j.
\] 

Our goal is to approximate $u$ by using finite element or finite differences.
The finite element approximation  on $V_h \subset V$ is based on the 
discrete solution operator $\calT_h:V_h\rightarrow V_h$ defined  by 
$\calT_hv_h:=w_h$ where $w_h$ is the unique function in
$V_h$ satisfying
\[
A(w_h,v_h)=(v_h,v_h),\Forall v_h\in V_h.
\]
The inverse of $\calT_h$ is denoted by $\calA_h$ and satisfies
\[
(\calA_h w_h,v_h)=A(w_h,v_h),\Forall v_h\in V_h.
\]

We obtain a ``semi-discrete'' approximation to $u$ of equation
\eqref{fracprob}
by defining  
\beq \label{algebraic}
u_h =\calT_h^\alpha \pi_h f := \calA_h^{- \alpha} \pi_h f,
\eeq
where 
$\pi_h$ is the $L^2(\Omega)$-orthogonal projection into $V_h$.  Note that $u_h$
can be expanded in the $L^2(\Omega)$-orthogonal  eigenfunctions of $\calT_h$, i.e.,  if
$\{\psi_{h,j},\mu_{h,j}\}$, for $j=1,\ldots, M$, denotes the eigenpairs and
$M$ is the dimension of $V_h$ then
\[
\pi_h f =\sum_{j=1}^{M} (f,\psi_{h,j})\, \psi_{h,j}
\]
and 
\beq u_h=\sum_{j=1}^{M} \mu_{h,j}^\alpha  (f,\psi_{h,j})\, \psi_{h,j}.
\label{seriesh}
\eeq
In this paper, we shall study a technique for approximating the solution
to \eqref{algebraic} which avoids computing the eigenvectors and
eigenvalues of $\calT_h$.

We note that the technique to be developed can be applied to finite
difference approximations as well.   
In this case, the discrete space is a finite dimensional space
of grid point values and $\pi_h f$ is replaced by the interpolant of $f$
at the grid nodes.  
The matrix $\calT_h$ comes from applying finite difference
approximations to the derivatives in the strong form (see,
\eqref{strong}) of problem \eqref{weak}.  Its scaling is not arbitrary   if we
expect $\calT_h^\alpha \pi_h f$ to converge to the grid point values of $f$.
This, in turn, implies the proper scaling for the discrete norms.

It has been shown in  \cite[Theorem 4.3]{BP-frac} that, if the operator $\calT$ 
satisfies elliptic regularity pickup with index 
$s \in  (0,1]$ (see Assumption \ref{assume1})
then  for appropriate $f$,
\[
\| u - u_h \| =\|\calT^{\alpha}f- \calT_h^{\alpha} \pi_h f\|  =O(h^{2s}).
\]  
See Section \ref{FEM} for more details. 
Here $\|\cdot\|$ denotes the $L^2(\Omega)$-norm.

Obviously, the discrete operator $\calA_h$ is symmetric and 
positive definite and the corresponding  matrix is full. 
We note that problems involving
finding approximations of  $\calA_h^{1/2}$, cf. \cite{Higham1997},
evaluating the sign function of $\calA_h$, cf. \cite{Kenney1991}, 
and other related functions of matrices have a long history in numerical linear algebra. 

\subsection{The idea of the method of Vabishchevich}

In this paper, we develop and study a method for approximating the
solution \eqref{algebraic}.  Our proposed 
method is related to an  idea of P. Vabishchevich,
\cite{Vabishchevich2015JCP}, which exhibits  $u_h$
as a solution of a special time dependent problem. Now we 
briefly explain the idea of his method.

We start with the
observation that the unique solution $\hu(t)$ to the ordinary differential
equation initial value problem,
\beq\bal 
\hu_t+ \frac {\alpha (\lambda-\delta)} {\delta +t (\lambda-\delta)} \hu = 0, \ \ \ 
\hu(0) = \delta^{-\alpha} \hat{v}
\eal
\label{ode}
\eeq
with $0<\delta<\lambda$ and $\alpha \ge 0$, is given by
\beq
\hu (t) = (\delta + t(\lambda-\delta))^{-\alpha}  \hat{v}
\label{i.0}
\eeq
and hence
\beq \hu (1)= \lambda^{-\alpha} \hat{v}.
\label{i.1}
\eeq
Note also that for $k>0$,
\beq
\hu(t+k) = [1+k(\lambda-\delta)/(\delta +t (\lambda-\delta))]^{-\alpha} \hu(t).
\label{i.2}
\eeq

Now suppose that the spectrum of $\calA_h$ is contained
in the interval $[\lambda_1, \lambda_M  ]$ 
with $\lambda_1>0$.   
Let $0<\delta<\lambda_1$ and set $\calB=\calA_h-\delta \calI$. We
consider the vector  valued ODE:  Find $U(t):[0,1]\rightarrow V_h$ satisfying 
\beq\bal 
U_t + \alpha \calB (\delta \calI+t \calB )^{-1} U & = 0,\\
U(0)  & =  \delta^{-\alpha} f_h 
\eal
\label{OP-ode}
\eeq
where $f_h:=\pi_h f$.  Expanding the solution to \eqref{OP-ode} as 
$$U(t)= \sum_{j=1}^{M} c_j(t) \psi_{j,h},$$ 
we find that $c_j(t)$ solves \eqref{ode} with $\lambda=\lambda_{j,h}$. 
Moreover, it follows from \eqref{i.0} and \eqref{i.2} that  
\beq
U(t):=(\delta \calI + t \calB)^{-\alpha}f_h \label{u(t)}
\eeq 
and 
\beq
U(t+k) = (\calI+k \calB (\delta \calI+t\calB)^{-1})^{-\alpha} U(t).
\label{evo}
\eeq
As proposed by
Vabischchevich \cite{Vabishchevich2015JCP}, it is then natural to consider numerical approximations to
\eqref{OP-ode} based on a time stepping method.

In \cite{Vabishchevich2015JCP}, Vabishchevich proposed a time stepping scheme based on
the backward Euler method and applied it to approximate fractional
powers of a discrete approximation $\calA_h$ 
of the Laplace
operator with homogeneous Dirichlet boundary conditions. The results of
numerical computations illustrating the accuracy, convergence, and
some theoretical aspects of the method were provided.


In this paper, we  take a different but related approach.   Instead
of approximating the solution of \eqref{OP-ode}, we simply
approximate the function $U(t_i)$ given by \eqref{u(t)} 
on an increasing sequence of nodes $0=t_0<t_1<\cdots< t_\Kind=1$.   
We start from the recurrence \eqref{evo}.
The ``time stepping'' methods that we
shall study  are based on
diagonal Pad\'e approximation to $(1+x)^{-\alpha}$, i.e.,
\beq
(1+x)^{-\alpha}\approx r_m(x) := \frac {P_m(x)}{Q_m(x)}
\label{pade}
\eeq
with $P_m$ and $Q_m$ being polynomials of degree $m$ and $Q_m(0)=1$.
The polynomials $P_m$ and $Q_m$ are then uniquely defined by requiring
that the first $2m+1$ terms of the Maclauren expansion of 
\[
(1+x)^{-\alpha}-r_m(x) 
\]
vanish.   The method that we
study and analyze is  given by setting  $U_0:=u(0)=\delta^{-\alpha}f_h$ and applying the recurrence
\beq
U_{l}=r_m(k_l \calB(\delta \calI+t_{l-1} \calB )^{-1}) U_{l-1},\quad
l=1,2,\ldots,\Kind
\label{recur}
\eeq
with $k_l=t_l-t_{l-1}$.  Here $U_l$ is our approximation of $U(t_l)$
so that $U_K $ approximates $U(1)=u_h$.
We shall see that these methods are unconditionally stable for
$m=1,2,\ldots$
and $\alpha \in (0,1)$.

Even though we take a different point of view, we are still solving \eqref{OP-ode}, 
as suggested by P. Vabishchevich, \cite{Vabishchevich2015JCP}.
It is important to note that although Problem \eqref{OP-ode} appears
harmless, it behaves considerably different than, for example, the classical parabolic
problem:
\beq
w_t + \calA_h w = 0.
\label{clas}
\eeq
For example,  if $w(0)= \psi_{j,h}$ then $w(1)= e^{-\lambda_{j,h}}  \psi_{j,h}$
while if $f_h=\psi_{j,h}$, the solution of \eqref{OP-ode}  is
$u(1)=\calA_h^{-\alpha} \psi_{j,h}  =\lambda_{j,h}^{-\alpha}\psi_{j,h}$.  This means that initial
time step errors in the high frequency components for our problem have much
stronger effect on the accuracy of the final solution. This is especially important
for problems whose solutions have minimal regularity. 

\subsection{Our contributions}

In this paper, we consider two time stepping schemes, one involving mesh
refinement near $t=0$ and the other using a fixed time step.
In both cases, we shall be using \eqref{recur}  to define our
solution but on different meshes in time.

The refinement scheme starts with an initial basic mesh with $t_0=0$, $t_i=2^{i-1-L} $, for
$i=1,2,\ldots L+1$ with $L$ chosen so that
$2^{-L}<(\lambda_M)^{-1}$. Subsequent finer meshes are defined by
partitioning each of the above intervals into $N$ equally spaced subdivisions.   The
refinement scheme leads to an error estimate
\[
\|\calA_h^{-\alpha}f_h -U_{N(L+1)}
\| \le C N^{-2m}  \|f_h\|
\] 
for the Pad\'e scheme  based on $r_m(x)$.    

The second scheme that we study is the simpler one using a fixed step
size $k_N=1/N$.  In this case, we obtain the error estimate
\beq
\|\calA_h^{-\alpha}f_h -U_{N}  
\| \le 
C k_N^{\alpha+\gamma}\|A_h^{\gamma}f_h\|,\quad \hbox{ for }0\le
\gamma \le 2m-\alpha.
\label{equalsp}
\eeq

It is clear that $L$ in the first method grows like the logarithm of 
$\lambda_M$ so that more steps are required by the refinement scheme
when the same $N$ is used in both.   However, in all of our numerical
examples, if one adjusts the values of $N$ in both schemes to obtain the
same absolute convergence, the refinement scheme requires less steps
overall. 

The question of when the norm on the right hand side of \eqref{equalsp}
can be controlled by natural norms on the data $f$ is open.  Although,
such a result for $\gamma \le 1$ was provided in \cite{wenyu-thesis},
the result for larger $\gamma$ is not known even in the finite element
case. We discuss this in more detail in Section \ref{FEM}. In fact, our numerical results 
in  Section \ref{numer}  suggest that the result is not true in general.

\section{Pad\'e Approximations}   

In this section, we develop diagonal Pad\'e
approximations to  $(1+x)^{-\alpha}$ for $\alpha\in (0,1)$ based on the
classical theory of Pad\'e approximations given by Baker \cite{Baker1975}.

Our approximations are of the form of \eqref{pade} with $m=1,2,\ldots$
and we shall write down explicit formula for the polynomials $P_m(x)$
and $Q_m(x)$. The starting point is the formula  \cite[relation (5.2)]{Baker1975}
or \cite[formula (2.1)]{Kenney1989}:
\[ 
(1+x)^{-\alpha} = \Fto (\alpha,1;1;-x).
\]
Here $\Fto (a,b;c;x)$ denotes the hypergeometric function defined by
\[
\Fto (a,b;c;x)=\sum_{j=0}^\infty \frac {(a)_j (b)_j} {j! (c)_j} x^j.
\]
Here $(a)_0=1$ and $(a)_j=a(a+1)\cdots(a+j-1)$ for $j>0$.
This series converges for $|x|<1$ provide that $c$ is not in
$\{0,-1,-2,\ldots\}$.  

Baker, \cite[formula (5.12)]{Baker1975}, also gives an explicit expression 
for the denominator:
\beq \bal 
Q_m(x)&= \Fto(-m,-\alpha-m; -2m;-x)\\
&:=\sum_{j=0}^m \frac
{(-m)_j\, (-\alpha-m)_j} 
{j!\, (-2m)_j} (-x)^j 
= 1 +\sum_{j=1}^m  a_j b_j(\alpha) x^j.
\eal
\label{den}
\eeq
Here $b_0(\alpha)=a_0=1$,
\[
b_j(\alpha)= (m+\alpha)((m-1)+\alpha)\cdots((m+1-j)+\alpha)
\]
and
\[a_j=\frac {m(m-1)\cdots(m+1-j)}{j! (2m(2m-1)\cdots(2m+1-j)}\qquad
\hbox{for }j=1,2,\ldots,m.
\]

Now, Theorem 9.2 of \cite{Baker1975} implies that $Q_m(x)/P_m(x)$ is the
diagonal Pad\'e approximation to $(1+x)^{\alpha}$ and again applying
(5.12) of \cite{Baker1975}, we find that 
\beq\bal 
P_m(x)&={}_2F_1(-m,\alpha-m; -2m;-x)\\
&:=\sum_{j=0}^m \frac {(-m)_j
	(\alpha-m)_j} {j! (-2m)_j} (-x)^j 
= 1 +\sum_{j=1}^m  a_j b_j(-\alpha) x^j.
\eal
\label{num}
\eeq
Using the above formulas, we find, for example:
\[
r_1(x)= \frac {1 + [(1-\alpha)/2]x }{ 1 + [ (1+\alpha)/2 ]x  }
\]
and 
\[
r_2(x)=\frac { 1 + [ (2-\alpha)/2]x +[(2-\alpha)(1-\alpha)/12]x^2 } { 1 +
	[(2+\alpha)/2]x+ [(2+\alpha)(1+\alpha)/12]x^2 }.
\]

For our further considerations, we need to discuss a relation
between the  denominators appearing in Pad\'e approximations and
orthogonal polynomials with respect to an appropriate weight $w$.  
We first note the series expansion of 
\[
(1-x)^{-\alpha} = \sum_{i=0}^\infty \frac {(\alpha)_i\, x^i} {i!}:=
\sum_{i=0}^\infty c_i x^i.
\]
The coefficients above satisfy
\beq \bal 
c_i = \frac 1 {\Gamma(\alpha)\Gamma(1-\alpha)} \int_0^1 x^i
\bigg ((1-x)^{-\alpha} x^{\alpha-1} \bigg ) \, dx 
:= \frac 1 {\Gamma(\alpha)\Gamma(1-\alpha)} \int_0^1 x^i
w(x)  \, dx, 
\eal
\label{coef}
\eeq
where $ w(x)=(1-x)^{-\alpha}x^{\alpha-1}$. By (7.7) of
\cite{Baker1975}, utilizing the fact that $P_m(-x)/Q_m(-x)$ is  the
Pad\'e approximation of $(1-x)^{-\alpha}$, we obtain that the denominator $Q_m(x)$ can be expressed by 
\[Q_m(x)=(-x)^m q_m(-1/x)\]
where $q_m$ is the monic polynomial of order $m$ which is orthogonal to the set of polynomials 
of degree less than $m$ with the weight $xw(x)$. As the roots of $q_m(x)$ are all in the interval $(0,1)$, those of
$Q_m(x)$ are in the interval $(-\infty,-1)$.

The expressions for the numerator and denominator in $r_m(x)$ imply the following
proposition.

\begin{proposition}  \label{p1}
	Let $\alpha$ be in $(0,1)$ and $m$ be a positive integer.
	Then, there are positive constants $\rho_m$ 
	satisfying
	\beq
	\rho_m\le r_m(x)\le 1,\qquad \hbox{for all }x\ge 0.
	\label{prop11}
	\eeq
\end{proposition}

\begin{proof}
	As $\alpha$ is in $(0,1)$, 
	$$
	0<b_j(-\alpha)<b_j(\alpha). $$
	This means for $x\ge 0$, 
	$a_j b_j(-\alpha)x^j\le a_j b_j(\alpha) x^j$ 
	with strict inequality when $x>0$.  The second inequality of  \eqref{prop11} follows by summation. 
	
	For the first inequality in \eqref{prop11}, we note that both $Q_m(x)$ and $P_m(x)$ are positive for
	$x\ge 0$ and 
	\[
	 \lim_{x\rightarrow \infty} r_m(x) = b_m(-\alpha)/b_m(\alpha)>0.
	 \]
	The first inequality in \eqref{prop11} immediately follows from the fact
	that $r_m(x)$ is continuous on $[-1,\infty)$ since all of the roots of
	$Q_m(x)$ are in $(-\infty,-1)$ and  takes values in $(0,1]$ for $x\ge 0$. 
\end{proof}

To clarify further the convergence of these approximations, we include
the following proposition.

\begin{proposition} \label{p2}  
	For $0\le s \le 2m+1$,
	\beq
	|(1+x)^{-\alpha} -r_{m}(x)| \le c_{m,s}   x^{s} \quad
	\hbox{for }x\in [0,\infty)\label{prop23}
	\eeq
	where 
	$$c_{m,s}=\max \{ Q_m(-1) 2^{s-2m},2^{1+s}\}.$$
\end{proposition}

\begin{proof}
	We use Theorem 5 of \cite{Kenney1989} which gives that 
	\[
	\bal (1+x)^{-\alpha}-r_m(x) &= \frac {Q_m(-1)}{Q_m(x)} 
	\sum_{n=2m+1}^\infty  \frac {(\alpha)_n(n-2m)_m}
	{n! (n+\alpha-m)_m} (-x)^n\\
	&:=\frac {Q_m(-1)}{Q_m(x)} (-x)^{2m+1} \sum_{i=0}^\infty d_{2m+1+i} (-x)^i.
	\eal
	\]
	A simple computation shows that $0<d_n<1$ for $n\ge 2m+1$ and so for $|x|<\tau<1$,
	\[
	 \bigg |\sum_{i=0}^\infty d_{2m+1+i} (-x)^i \bigg |\le (1-\tau)^{-1}.
	 \]
	Thus,  for $x\in [0,\tau]$ and $s\in [0,2m+1]$,
	\[
	\begin{aligned}
	\big |r_m(x)-(1+x)^{-\alpha} \big | & \le |Q_m(-1)| (1-\tau)^{-1} x^{2m+1} \\
	& \le |Q_m(-1)| \tau^{2m+1-s} (1-\tau)^{-1} x^{s}.
	\end{aligned}
	\]
	Finally, for the case $ x \in[\tau, \infty) $  by  Proposition \ref{p1},
	\[
	 |r_m(x)-(1+x)^{-\alpha}|\le 2\le 2x^{s}/\tau^{s}, \quad \hbox{
		for } x\in [\tau,\infty)
	\]
	and \eqref{prop23} follows by  taking $\tau=1/2$.
\end{proof}

\section{The time stepping schemes and their analysis}

In this section, we define and analyze both equally spaced time stepping 
schemes as well as schemes employing refinement near the origin.    
We shall restrict ourselves to the approximating the solution to finite dimensional 
problem \eqref{algebraic} described in the introduction even though generalizations to hermitian
and non-hermitian bounded operators on infinite dimensional spaces are possible.  
Recall that $\calB=\calA_h-\delta \calI$ with $\delta\in
(0,\lambda_1)$ and that $\| \cdot\|$ and $(\cdot,\cdot)$ denote, respectively, 
the norm and inner product in $L^2(\Omega)$.

\subsection{Time-stepping method on geometrically refined meshes (GRM)}
We  first consider the geometrically refined mesh that is constructed in two steps.

First, we take 
\[
\Lnt=\lceil \log(\lambda_M)/\log(2) \rceil
\]
and set $t_i=2^{i-1-L}$ for $i=1,\ldots,\Lnt+1$ and $t_0=0$. Note that $t_{\Lnt+1}=1$ and 
$2^{-\Lnt}\le (\lambda_M)^{-1}$.
Next, we define 
\begin{equation}\label{mesh-t}
k_n= \left \{ \bal t_{n}/N&:\qquad \hbox{for }n=1,\ldots,\Lnt,\\
t_1/N&:\qquad \hbox{when }n=0.\eal
\right . 
\end{equation}
Note that $k_0=k_1$ and 
\beq  {k_0 \lambda_M} = \frac {2^{-\Lnt}
	\lambda_M}{N}\le N^{-1}.
\label{Mb}
\eeq
The grid that we use in our computations is obtained by partitioning 
each subinterval 
$I_n=[t_n,t_{n+1}]$,    $ n=0,1,\ldots, \Lnt$,
into $N$ subintervals with end points
\[
t_{n,j} := t_{n} + jk_n  \quad \hbox{ for } j=0,\ldots,N.
\]
It can be seen that $t_{n,0}=t_{n}$ and $t_{n,N}=t_{n+1}$, $ n=0,1,\ldots, \Lnt$,
so that the mesh has totally $(\Lnt+1)N$ intervals.

Based on this partitioning, the refined time stepping method for
approximating $A_h^{-\alpha}v$ for $v\in V_h$ is given by:

\begin{algorithm}
	\label{a-basic}
	\begin{algorithmic}
		\STATE (a) Set $U_{0}=\delta^{-\alpha} \ff$. 
		 \STATE(b) {For $n=0,1,\ldots,L:$}
			\INDSTATE[1.5] (i) Set $U_{n,0}=U_{n}$.
			\INDSTATE[1.5] (ii) For {$j=1,2,\ldots,N$}, set 
				   \INDSTATE[3]  $$U_{n,j}=r_m(k_n \calB(\delta \calI+t_{n,j-1} \calB)^{-1})U_{n,j-1}.$$
		   \INDSTATE[1.5] (iii) Set $U_{n+1}=U_{n,N}$.
	\end{algorithmic}
\end{algorithm}
After executing the above algorithm, $U_{\Lnt+1}$ is the approximation to
$\calA_h^{-\alpha} \ff$.  Note that the notation differs slightly from
that used in the introduction.  The computation of $U_{\Lnt+1}$ requires
$\Kind=(\Lnt +1)N$
time steps.

The discrete eigenvalues and eigenvectors will play a major role in our
analysis so, for notational simplicity, we denote them by
$\{\psi_j,\lambda_j\}$  (instead of $\{\psi_{j,h},\lambda_{j,h}\}$ as in
the introduction).
We then have 
\[
\ff=\sum_{j=1}^\cM (\ff,\psi_j) \psi_j.
\]
Moreover,
\[
\calA_h^{-\alpha} v =\sum_{j=1}^\cM \lambda_j^{-\alpha} (\ff,\psi_j) \psi_j.
\]
The expansion for $U_{L+1}$ is given by
\[
U_{L+1}=\sum_{j=1}^\cM \mu(\lambda_j)  (\ff,\psi_j) \psi_j
\]
where the coefficient $\mu(\lambda_j)$ is given by the following algorithm:

\begin{algorithm}
	\label{a1}
	\begin{algorithmic}
			\STATE (a) Set $\mu_{0}=\delta^{-\alpha}$. 
			\STATE (b) For $n=0,1,\ldots,L:$
					\INDSTATE[1.5] (i) Set $\mu_{n,0}=\mu_{n}$.
					\INDSTATE[1.5] (ii) For $j=1,2,\ldots,N$, set 
					$$\mu_{n,j}=r_m(k_n (\lambda-\delta)/(\delta+t_{n,j-1} (\lambda-\delta)))\mu_{n,j-1}.$$
					\INDSTATE[1.5] (iii) Set $\mu_{n+1}=\mu_{n,N}$.
			\STATE (c) Set $\mu(\lambda)=\mu_{\Lnt+1}$.
	\end{algorithmic}
\end{algorithm}

\begin{theorem}\label{rate4nonsmooth}
	Let $N$ be a positive integer.   Then
	\begin{equation}\label{t1}
	|\lambda^{-\alpha} - \mu(\lambda)| \le \widetilde c N^{-2m}, \Forall
	\lambda\in [\lambda_1,\lambda_M]
	\end{equation}
	and 
	\begin{equation}\label{t2}
	\|\calA_h^{-\alpha} \ff -U_{\Lnt +1}\| \le \widetilde c N^{-2m} \| \ff \|, \Forall  \ff \in V_h.
	\end{equation}
	Here $\widetilde c$ is a
	constant depending only on $\delta$, $m$ and $\alpha$.
\end{theorem}

\begin{proof}  Fix $\lambda$ in  $[\lambda_1,\lambda_M]$ and let
	$\{\mu_n,\mu_{n,j}\}$ be as in Algorithm \ref{a1}. Further,
	for $n=0,\ldots,\Lnt$ and $j=1,\ldots,N$ let 
	$$v_{n,j}=(\delta +t_{n,j}(\lambda-\delta))^{-\alpha}, \ \ 
	e_{n,j}=v_{n,j}-\mu_{n,j}, \ \mbox{ and } \ \ \theta_{n,j}= k_n (\lambda-\delta)
	/(\delta+t_{n,j-1}(\lambda-\delta)).
	$$ 
	We note that as in \eqref{evo}, 
	\begin{equation*}
	v_{n,j}=\left(1+\theta_{n,j}\right)^{-\alpha}v_{n,j-1}.
	\end{equation*}
	Thus,
	\[
	e_{n,j} = r_m(\theta_{n,j}) 
	e_{n,j-1} + [(1+\theta_{n,j})^{-\alpha}-r_m(\theta_{n,j})] v_{n,j-1}
	\]
	and  by Proposition \ref{p1} and \ref{p2},
	\begin{equation}\label{start}
		|e_{n,j}| \le 
		|e_{n,j-1}| +
		|(1+\theta_{n,j})^{-\alpha}-r_m(\theta_{n,j})|\,|v_{n,j-1}|,
	\end{equation}
	for $n=0,1,\ldots,L$ and $j=1,\ldots,N$.
	
	Using  $e_{0}=0$, \eqref{start},  Proposition \ref{p2} and the
	definition of $v_{n,j-1}$,
	we have 
	\begin{equation}\label{err}
	\begin{aligned}
	|\lambda^{-\alpha}-\mu(\lambda)| &= |e_{\Lnt+1}|=\sum_{n=0}^{\Lnt}\sum_{j=1}^{N} \bigg(
	|e_{n,j}|-|e_{n,j-1}|\bigg)\\
	&\le c_{m,2m+1}\sum_{n=0}^{\Lnt}\sum_{j=1}^{N}  \theta_{n,j}^{2m+1}
	|v_{n,j-1}|\\ 
	&= c_{m,2m+1}\sum_{n=0}^{\Lnt}\sum_{j=1}^{N} \frac{(k_n
		(\lambda-\delta))^{2m+1} }{(\delta+t_{n,j-1}(\lambda-\delta))^{2m+\alpha+1}}.
	\end{aligned}
	\end{equation}
	
	We note that by \eqref{Mb},
	\beq\bal 
	\sum_{j=1}^N \frac{(k_0
		(\lambda-\delta))^{2m+1}}{(\delta
		+t_{0,j-1}(\lambda-\delta))^{2m+\alpha+1}}
	& \le \delta ^{-2m-\alpha-1}\sum_{j=1}^N N^{-2m-1} =\delta^{-2m-\alpha-1} N^{-2m}.
	\eal
	\label{Mterm}
	\eeq
	The remaining terms in \eqref{err} will be
	bounded by integration. Applying mean value theorem for integration it
	follows that for some $t_\theta\in [t_{n,j-1},t_{n,j}]$,
	\begin{equation*}
	\begin{aligned}
	&\int_{t_{n,j-1}}^{t_{n,j}}  \frac
	{t^{2m} (\lambda-\delta)^{2m}}
	{(\delta+t (\lambda-\delta))^{2m+\alpha+1}}\, dt= \frac
	{t_\theta^{2m} (\lambda-\delta)^{2m}}
	{(\delta+t_\theta (\lambda-\delta))^{2m+\alpha+1}}k_n\\
	&=\left(\frac{t_\theta}{t_{n,j-1}}\right)^{2m}\left(\frac{\delta+t_{n,j-1} (\lambda-\delta)}{\delta+t_{\theta} (\lambda-\delta)}\right)^{2m+\alpha+1} \frac
	{t_{n,j-1}^{2m} (\lambda-\delta)^{2m}}
	{(\delta+t_{n,j-1} (\lambda-\delta))^{2m+\alpha+1}}k_n\\
	&\ge \left(\frac{t_{n,j-1}}{t_\theta} \right)^{\alpha+1}\frac
	{t_{n,j-1}^{2m} (\lambda-\delta)^{2m}}
	{(\delta+t_{n,j-1} (\lambda-\delta))^{2m+\alpha+1}}k_n\\
	&\ge \left(\frac N {N+1} \right)^{\alpha+1}\frac
	{t_{n,j-1}^{2m} (\lambda-\delta)^{2m}}
	{(\delta+t_{n,j-1} (\lambda-\delta))^{2m+\alpha+1}}k_n.
	\end{aligned}
	\end{equation*}
	Now  $(N+1)/N \le 2$ and $t_{n}\le t_{n,j-1}$ for $j=1,2,\ldots,N$ so
	that
	\beq
	\bal
	\frac{(k_n
		(\lambda-\delta))^{2m+1} }{(\delta+t_{n,j-1}(\lambda-\delta))^{2m+\alpha+1}}
	&=N^{-2m}\frac{ t_{n}^{2m} k_n  (\lambda-\delta)^{2m+1}}
	{(\delta+t_{n,j-1} (\lambda-\delta))^{2m+\alpha+1}}\\
	&\hskip -80pt\relax \le  2^{\alpha+1} N^{-2m}   (\lambda-\delta)\int_{t_{n,j-1}}^{t_{n,j}}  \frac
	{t^{2m} (\lambda-\delta)^{2m}}
	{(\delta+t (\lambda-\delta))^{2m+\alpha+1}}\, dt.
	\eal
	\label{onet}
	\eeq
	
	Applying \eqref{onet} to the remaining terms in \eqref{err} we have
	\begin{equation}\label{nterm}
	\begin{aligned}
	&\sum_{n=0}^{\Lnt}\sum_{j=1}^{N}\bigg(
	|e_{n,j}|-|e_{n,j-1}|\bigg) \\
	&\le \delta^{-2m-\alpha-1} N^{-2m}+ 2^{\alpha+1}N^{-2m} 
	(\lambda-\delta)\int_0^1  \frac
	{t^{2m} (\lambda-\delta)^{2m}}
	{(\delta+t (\lambda-\delta))^{2m+\alpha+1}}\, dt\\
	&\le  \delta^{-2m-\alpha-1} N^{-2m}+2^{\alpha+1} N^{-2m} \int_0^\infty  \frac {z^{2m}} {(\delta+z)^{2m+\alpha+1}}\, dz.
	\end{aligned}
	\end{equation}
	As the last integral on the right converges,  \eqref{t1} follows
	combining  \eqref{err}, \eqref{Mterm} and \eqref{nterm}.
	
	The inequality \eqref{t2} follows immediately from \eqref{t1} and  the Parseval's identities:
	\begin{displaymath}
	\begin{aligned}
	\| \calA_h^{-\alpha} \ff -U_{\Lnt+1}\|^2&=\sum_{j=1}^\cM
	|\lambda_j^{-\alpha}-\mu(\lambda_j)|^2 |( \ff,\psi_j)|^2 \\
	&\le \widetilde c^2 N^{-4m}
	\sum_{j=1}^\cM  |(v,\psi_j)|^2=\widetilde c^2 N^{-4m}  \| \ff \|^2.
	\end{aligned}
	\end{displaymath}
\end{proof}

\begin{remark}  \label{s-h-norm} For any $s\in \RR$ and $v\in V_h$, let
	\begin{equation}\label{hs-norm}
		\| \ff \|_{{s,h}}:=\bigg(\sum_{j=1}^\cM \lambda_j^{s}
		|( \ff,\psi_j)|^2\bigg)^{1/2}=\|\calA_h^{s/2} \ff\|.
	\end{equation}
	Then,  it follows from the proof of the above theorem that for any
	$s\in \RR$, 
	\[
	\|\calA_h^{-\alpha} \ff   -  U_{\Lnt+1}  \|_{s,h} \le \widetilde c N^{-2m} \| \ff \|_{s,h}.
	\]
\end{remark}


\subsection{Time-stepping method on uniform meshes (UM)}
We next consider uniform time stepping.  In this case, given a positive
integer $N$, we set $k_N=1/N$ and $t_n=nk_N$. The approximation is
obtained from the recurrence
\beq
\bal U_0&=\delta^{-\alpha} \ff \in V_h,\\
U_n&= r_m(k_N \calB (\delta \calI +t_{n-1} \calB)^{-1}) U_{n-1},\quad \hbox{ for }
n=1,2,\ldots,N.
\eal
\label{unif}
\eeq
In this case, $U_N$ is our approximation to $\calA_h^{-\alpha} \ff $.  The
analysis of the error requires the following proposition.

\begin{proposition}\label{p-est-rm}  For $\lambda\ge \delta$,
	set
	\beq
	\theta_n= \frac {k_N (\lambda-\delta)}
	{\delta+t_{n-1}(\lambda-\delta)}.
	\label{thnunif}
	\eeq
	Then for $q\ge p>1$,  
	\begin{equation}\label{est-rm}
	\prod_{n=p}^{q} r_m(\theta_n)\le c \prod_{n=p}^{q} (1+\theta_n)^{-\alpha},
	\end{equation} 
	with $c$ depending only on $m$ and $\alpha$.
\end{proposition}

\begin{proof}
	We note that for $n>1$, $\theta_n\le (n-1)^{-1}\in [0,1]$.
	By Proposition  \ref{p2},  
	\begin{displaymath}
		\begin{aligned}
		r_m(\theta_n)&\le
		(1+\theta_n)^{-\alpha}+c_{m,2m+1}|\theta_n|^{2m+1}\\
		&\le
		(1+\theta_n)^{-\alpha} (1+2^{\alpha}c_{m,2m+1} (n-1)^{-2m-1})
		\end{aligned}
        \end{displaymath}
	and hence
	\begin{displaymath}
		\frac { r_m(\theta_n)}{(1+\theta_n)^{-\alpha}}\le (1+2^{\alpha}c_{m,2m+1}
		(n-1)^{-2m-1}) .
	\end{displaymath}
	Thus,  for $q\ge p>1$, 
	\begin{displaymath}
	\prod_{n=p}^{q} \frac{r_m(\theta_n)}{(1+\theta_n)^{-\alpha}}\le
	\prod_{n=p}^{q}  (1+2^{\alpha}c_{m,2m+1}(n-1)^{-2m-1})  \le c
	\end{displaymath}
	with
	\begin{displaymath}
		c=  \prod_{j=1}^{\infty}  (1+2^{\alpha}c_{m,2m+1}
		j^{-2m-1}).
	\end{displaymath}
\end{proof}

Theorem~7.2 of \cite{thomeeBook} provides error estimates for single
step approximations for the standard parabolic problem \eqref{clas} with non-smooth
initial data.     
The next theorem
has the same flavor however differs
significantly as the solutions of our problem exhibit  less regularity.

\begin{theorem}\label{t:uniform1}
	Let $N>1$, $\ff  \in V_h$ and $U_N$ be defined by \eqref{unif}.  Then,
	for $\gamma\ge0$ and  $\alpha+\gamma\le 2m$,
	\beq \| \calA_h^{-\alpha} \ff -  U_N \| \le \widetilde c k_N^{\alpha +\gamma} \|\calA_h^\gamma \ff \|
	\label{unif-b}
	\eeq
	with $\widetilde c$ depending only on $\alpha$, $m$,  $\gamma$, and $\delta$.
	As in Remark \ref{s-h-norm}, the left hand norm above can be replaced by
	$\|\cdot\|_{h,r}$ provided that the right is replaced by
	$\|\calA_h^\gamma \ff\|_{h,r}$.
\end{theorem}

\begin{proof}   In this proof, $c$ denotes a generic positive constant
	only depending on $\alpha$,
	$\gamma$, $m$ and $\delta$.    We fix $\lambda\ge \delta$ and define, 
	for $j\ge l\ge 1$, 
	\[
	r_m^{j,l}(\lambda):=r_m^{j,l}=r_m(\theta_j)\,r_m(\theta_{j-1})\,\cdots\,r_m(\theta_l)
	\]
	and 
	\[w^{j,l}(\lambda):=w^{j,l}=(1+\theta_j)^{-\alpha}\,(1+\theta_{j-1})^{-\alpha}\,\cdots\,(1+\theta_l)^{-\alpha}.
	\]
	Finally, we set
	\[e_N(\lambda):=e_N=(w^{N,1}-r_m^{N,1}) \delta^{-\alpha}.
	\]
	
	We note that it is a consequence of \eqref{i.1} and  \eqref{i.2} that 
	\beq
	\delta^{-\alpha}  w^{N,1} =
	\lambda^{-\alpha}. 
	\label{prodid}
	\eeq
	For any $j>1$,
	\[
	w^{j,1}-r_m^{j,1}=  [(1+\theta_j)^{-\alpha}-r_m(\theta_j)] w^{j-1,1}
	+r_m(\theta_j) [w^{j-1,1}-r_m^{j-1,1}].
	\]
	Repeated application of this identity leads to 
	\beq
	\bal  e_N  & = \delta^{-\alpha} [w^{N,1}-r_m^{N,1}] \\
	& =   \delta ^{-\alpha}[(1+\theta_N)^{-\alpha} - r_m(\theta_N)] w^{N-1,1} \\
	&\qquad +
	\delta ^{-\alpha} r_m(\theta_N) [(1+\theta_{N-1})^{-\alpha} - r_m(\theta_{N-1})] w^{N-2,1}\\
	&\qquad  + \cdots  \\
	&\qquad + \delta ^{-\alpha}  r_m^{N,2} [(1+\theta_{1})^{-\alpha} -  r_m(\theta_{1})].  
	\eal
	\label{teles}
	\eeq
	
	Proposition \ref{p-est-rm}  implies that for $j\ge 2$,
	\beq
	|r_m^{N,j}|\le c w^{N,j}.\label{rb}
	\eeq
	We first bound the last term of  \eqref{teles} by applying this
	and \eqref{prodid} to obtain
	\beq
	\bal 
	T_1:=\delta^{-\alpha}|r_m^{N,2} [(1+\theta_{1})^{-\alpha} -
	r_m(\theta_{1})]| & \le c \delta^{-\alpha} w^{N,2}| (1+\theta_{1})^{-\alpha} -
	r_m(\theta_{1})|\\
	&= c\lambda^{-\alpha} (1+\theta_1)^\alpha |(1+\theta_{1})^{-\alpha} -
	r_m(\theta_{1})|.
	\eal
	\label{t1b0}
	\eeq
	When $\theta_1\le 1$,  Proposition  \ref{p2} with $s=\alpha+\gamma$ gives,
	\[
	T_1 \le c \theta_1^{\alpha+\gamma} \lambda^{-\alpha}.
	\]
	When $\theta_1>1$, since $\gamma\ge 0$,
	\[T_1 \le c (1+\theta_1)^\alpha   \lambda^{-\alpha}
	\le c \theta_1^{\alpha+\gamma} \lambda^{-\alpha}.
	\]
	Thus, in either case, since $\theta_1 = k_N (\lambda-\delta)/\delta \le
	k_N \lambda/\delta$,
	\beq
	T_1\le c k_N^{\alpha+\gamma} \lambda^{\gamma}.
	\label{t1b}
	\eeq
	
	The absolute value of the other terms in \eqref{teles} are given by
	\[
	T_j:=| r_m^{N,j+1}  [(1+\theta_{j})^{-\alpha} - r_m(\theta_{j})]
	w^{j-1,1}|,\qquad 
	\hbox{ for } j=2,3,\ldots N,
	\]
	where we have defined $r_m^{N,N+1}=1$ for convenience of notation.
	Similar to \eqref{t1b0}, we have 
	\[
	T_j \le c \lambda^{-\alpha} (1+\theta_j)^\alpha |(1+\theta_{j})^{-\alpha} -
	r_m(\theta_{j})|.
	\]
	In this case, $\theta_j$ is in $[0,1]$
	and we apply Proposition \ref{p2} with
	$s=1+\alpha+\gamma$ to obtain
	\[
	T_j\le c \lambda^{-\alpha} \theta_j^{1+\alpha+\gamma}\le
	c \frac {k_N (\lambda-\delta) k_N^{\alpha+\gamma} \lambda^{\gamma}}
	{(\delta+t_{j-1} (\lambda-\delta))^{1+\alpha+\gamma}}.
	\]
	Thus,  
	\beq
	\sum_{j=2}^N  T_j \le 
	c (\lambda-\delta) k_N^{\alpha+\gamma}
	\lambda^{\gamma}
	\int_0^1(\delta+t
	(\lambda-\delta))^{-1-\alpha-\gamma}\, dt
	\le c  k_N^{\alpha+\gamma}
	\lambda^{\gamma}.
	\label{t2b}
	\eeq
	Combining \eqref{t1b} and \eqref{t2b} gives
	\[
	|e_N(\lambda)| \le c  k_N^{\alpha+\gamma} \lambda^{\gamma}.
	\]
	
	We note that by \eqref{prodid},
	\[
	\calA_h^{-\alpha} \ff = \delta^{-\alpha} \sum_{j=1}^\cM w^{N,1}(\lambda_j)
	( \ff,\psi_j)\psi_j
	\quad\; \mbox{and} \quad\; 
	U_N= \delta^{-\alpha} \sum_{j=1}^\cM r_m^{N,1}(\lambda_j)
	( \ff,\psi_j)\psi_j .
	\]
	Thus,
	\[
	\bal
	\| \calA_h^{-\alpha}  \ff -U_N\|^2 &= \sum_{j=1}^\cM |e_N(\lambda_j)|^2
	|( \ff,\psi_j)|^2\\
	& \le  c k_N^{2\alpha+2\gamma} \sum_{j=1}^\cM
	\lambda_j^{2\gamma} |(v,\psi_j)|^2 = c k_N^{2\alpha+2\gamma}
	\| \calA_h^\gamma \ff \|^2
	\eal
	\]
	and this completes the proof.
\end{proof}

As seen in the above theorem, the discrete regularity of the 
solution determines the rate of convergence for the uniform step size  time stepping
method.  To some extent, the regularity of the discrete solution is related to the
regularity properties of the continuous problem which is being
approximated. This will be discussed in the next section.  

\section{Finite element approximation to fractional powers of second order elliptic operators}
\label{FEM}

We start with the second order elliptic problem  associated with the
bilinear form \eqref{bi-form} of the introduction, namely 
the  boundary value problem:  
\beq\label{strong}
\bal
-\nabla \cdot (a(x)\nabla w) + q(x) w &=f, \quad \hbox{ for }x\in \Omega,\\
w(x)&=0,\quad \hbox{ for }x\in \partial \Omega.
\eal
\eeq
Here $q(x),a(x)$ and $\Omega$ are as in the introduction.  The bilinear
form \eqref{bi-form} results from \eqref{strong} in the usual way, i.e.,
integration against a test function and integration by parts.

We
start by providing some results for the error between the
semi-discrete approximation $u_h$ given by \eqref{algebraic} and the
solution $u$ of \eqref{fracprob}.    These results depend on the following
regularity condition:

\begin{assumption}\label{assume1}  
	$\calT$ satisfies elliptic regularity pickup with index $s \in (0,1]$,
	that is 
	\begin{enumerate}[(a)]
		\item For $f\in H^{-1+s}(\Omega)$, $\calT f$ is in $H^{1+s}(\Omega)$ and
		there is a constant $c$ not depending on $f$ satisfying
		\[
		\|\calT f \|_{H^{1+s}(\Omega)} \le c \|f\|_{H^{-1+s}(\Omega)}.
		\]
		\item    $\calA=\calT^{-1}$ is a bounded map of $H^{1+s}(\Omega) $ into
		$H^{-1+s}(\Omega)$.
	\end{enumerate}
\end{assumption}

\begin{remark}  \label{norms1s}
	In \cite{BP-frac}, it has been shown that this implies
	\[
	D(\calA^{t/2}) =H^1_0(\Omega)\cap H^t(\Omega), \quad \hbox{ for  } t \in[1,1+s]
	\]
	with equivalent norms.
\end{remark}

The above remark shows that $D(\calA^{t/2})$ coincides with a Sobolev space
of index $t$.   Accordingly, we introduce the notation 
\[
\dH^t=D(\calA^{t/2}), \quad \hbox{ for } t\ge 0.
\]

A detailed estimation of the error $u - u_h= \calT^{\alpha}f- \calT_h^{\alpha} \pi_h f$
can be found  in \cite[Theorem 4.3]{BP-frac} and is  summarized below
(see also, \cite{fujitasuzuki} for the case when $s=1$).
\begin{theorem}\label{FEM-error} 
	(see, \cite[Theorem 4.3]{BP-frac}) Let Assumption \ref{assume1}
	hold and assume that the mesh is globally quasi-uniform so that the 
	inverse inequality holds, e.g. \cite{thomeeBook}
	or \cite[inequality (45)]{BP-frac}.
	Set $\beta =s -\alpha$ when $s > \alpha$ and $\beta=0$ when $s \le \alpha$.
	For $\gamma  \ge \beta$, there is a constant $C$ uniform in $h$ such that
	\beq\label{BP-error}
	\|\calT^{\alpha}f- \calT_h^{\alpha} \pi_h f \|   \le C_{h,\gamma}h^{2s}  \|f\|_{\dH^{2\gamma}}
	\Forall f \in \dH^{2 \gamma},
	\eeq
	where
	\[
	  C_{h,\gamma} = \left \{  \bal
	C \log(1/h),  & \quad \hbox{ when } \  \gamma = \beta \ \ \mbox{and}  \ s \ge \alpha,\\
	C,  & \quad \hbox{ when } \  \gamma > \beta \ \ \mbox{and}  \ s \ge \alpha ,\\
	C, & \quad \hbox{ when } \ \alpha > s. \eal
	\right.
	\]
\end{theorem}

As we see from this theorem, the rate of convergence in the
$L^2(\Omega)$-norm is 
the result of 
an interplay between the fractional order $\alpha$, the regularity 
pick up $s$ of the solution of problem  \eqref{fracprob}, and the 
regularity of the right hand side $f$.  The bottom line is that one recovers optimal 
convergence rate $O(h^{2s})$ for $\alpha > s$ when $f \in L^2(\Omega)$.  
However, if $\alpha < s$, the solution is not in $H^{2s}(\Omega)$
without  extra regularity from $f$ so this additional smoothness is
needed  to get the same rate.

Now if we approximate the problem \eqref{algebraic} using the method of Vabishchevich 
on the geometrically refined mesh as described in  Algorithm \ref{a-basic}, we get the
following bound for the total error (approximation by finite elements and time-stepping):
\begin{corollary}  \label{c:geometric}
	Assume the conditions of Theorem \ref{FEM-error} hold.
	Then  $U_{L+1}$, obtained by  Algorithm \ref{a-basic} with $v=\pi_h f$,
	$\calB=\calA_h-\delta \calI$ and
	performing $K=(L+1)N$ steps, satisfies 
	\[
	\| \calT^\alpha f - U_{L+1} \| \le C (h^{2s} \|f \|_{\dH^{2\gamma}} + N^{-2m} \|f \|).
	\]
\end{corollary}

\begin{corollary}\label{c:uniform}
	Assume the conditions of Theorem \ref{FEM-error} hold. Then
	$U_{N}$, obtained by applying \eqref{unif} with $v=\pi_h f $ and
	$\calB=\calA_h - \delta \calI$, 
	satisfies 
	\[
	\| \calT^\alpha f - U_{N} \| \le C (h^{2s}\|f \|_{\dH^{2\gamma}}
	+ N^{-\alpha-\beta}\|\calA_h^\beta \pi_h f\|), \quad \hbox{for
	}0\le\beta\le 2m-\alpha.
	\]
\end{corollary}

We next consider the question of bounding the norm $ \|\calA_h^\beta\pi_h f\|$ in 
terms of the regularity of $f$.   For $\beta \in [0,1/2]$, this reduces to showing 
that the $L^2(\Omega)$-projector into $V_h$ is a bounded operator on $H^1_0(\Omega)$ 
with bound independent of $h$. For globally quasi-uniform meshes, this result is given in
\cite{bankdupont,BrambleXu} while the case of certain refined meshes is given in 
\cite{bankYserentant}. When the $H^1$ bound holds, by interpolation, there is a constant $c$
depending only on $\beta \in [0,1/2]$ satisfying 
\beq
\|\calA_h^\beta\pi_h v \| \le c \|\calA^\beta v\|, \Forall v\in D(\calA^\beta).
\label{discretepih}
\eeq

We extend the above inequality to $\beta\in [1/2,(1+s)/2]$ in the next lemma whose
proof is included for completeness as it was already observed in
\cite{wenyu-thesis}.

\begin{lemma} Assume that Assumption \ref{assume1} holds and that the
	mesh is globally quasi-uniform.   Then there is a constant $c$  depending on
	$\beta\in [1/2,(1+s)/2]$ such that \eqref{discretepih} holds.
\end{lemma}

\begin{proof} 
	Let $P_h:H^1_0(\Omega)\rightarrow V_h$ denote the elliptic
	projector, i.e., $P_hw=w_h\in V_h$ is the unique solution of 
	\[
	A(w_h,\theta_h)=A(w,\theta_h),\Forall \theta_h\in V_h.
	\]
	Without loss of generality, we can take the norm on $H^1_0(\Omega)$ to
	be 
	\[\|v\|_{H^1(\Omega)} = A(v,v)^{1/2}=\|\calA^{1/2} v\|.\]
	We then have  
	\[
	\| \calA_h^{1/2}P_h w\|=  \|P_h w\|_{H^1(\Omega)} \le
	\|w\|_{H^1(\Omega)}=\|\calA^{1/2}w\|,
	\]
	for all $ w\in H^1_0(\Omega)=D(\calA^{1/2})$, while the identity 
	$\calA_h P_h v=\pi_h \calA v$ for $v\in D(\calA)$ implies that
	\[\|\calA_h P_h v \|= \|\pi_h \calA v \| \le \|\calA v\|,\Forall v\in D(\calA).\]
	It follows by interpolation that for $r\in [1/2,1]$,
	\[\|\calA^r_hP_h v \|\le   \|\calA^rv\|, \Forall v\in D(\calA^r).\]
	Now for $t\in [1,1+s]$, Remark \ref{norms1s} implies that for $
	\dH^{t}=D(\calA^{t/2})=H^t(\Omega)\cap H^1_0(\Omega)$.  
	Thus, for $v\in \dH^{t}$ ,  
	\[
	\|\pi_h v \|_{h,t} \le 
	\|(\pi_h-P_h)v\|_{h,t} +\|P_h v\|_{h,t}\le 
	C (h^{1-t} \|(\pi_h-P_h)v\|_{H^1(\Omega)} +\|v\|_{\dH^{t}(\Omega)}) 
	\]
	where we also used the inverse inequality for the last inequality
	above.   The inequality \eqref{discretepih} for $\beta = t/2\in
	[1/2,(1+s)/2]$  follows from  the above inequality,
	the triangle inequality and the well know error estimates
	\begin{displaymath}
		\|(I-\pi_h)v\|_{H^1(\Omega)}+\|(I-P_h)v\|_{H^1(\Omega)} \le ch^{t-1}
		\|v\|_{H^{t}(\Omega)}.
	\end{displaymath}
\end{proof}

\section{Numerical examples} \label{numer}

In this section, we present numerical examples for the problem 
coming from \eqref{bi-form} with $a(x)=1$, $q(x)=0$ and
$\Omega\subset\RR^d$, for $d=1,2$.

\subsection{One dimensional examples:  $\Omega=(0,1)$}
We will consider approximating the solution of  \eqref{fracprob}  with the following
choices of $f$:
\begin{enumerate}[(a)]
	\item $f= \hbox{exp}(-1/x-1/(1-x)+4) $  so that $f\in \dH^{s}$
	for any $s$.
	\item $f= x(1-x)  $  so that $f\in \dH^{s}$ for $s<\frac52$.
	\item $f=\min(x,1-x)  $  so that $f\in \dH^{s}$ for $s<\frac32$.
	\item $f=1$, so that $f\in \dH^{s}$ for $s<\frac12$.
\end{enumerate}
We note that $f=1$ fails to be in $\dH^{1/2}$ since functions in
$\dH^{1/2}$ vanish at $x=0$ and $x=1$.

The first set of runs demonstrate the time stepping error behavior using the
geometric refined time stepping algorithm (GRM) and the uniform time
stepping scheme (UM) for various $\alpha$ and $N$.  In this case, we use
a fixed equally spaced mesh with $h=1/1000$ and 
$L=\lceil 2  {|\log h |}/{\log 2}\rceil$.    
The total number of time
steps for the GRM scheme is thus $(L+1)N$, with $N=1,2,4,8,\ldots$.  To
make the comparison more meaningful, we report the errors obtained using
the UM and GRM algorithms  as a function of the number of solves.

For the first plot, we use $f$ given by (c) above and $\alpha=0.1,0.5,0.9$.  
Figure \ref{figure1d} gives  $\log-\log$ plots of
the relative $L^2$ error between $u_h=\mathcal{A}_h^{-\alpha} \pi_h f$ and the
result obtained using the refined and uniform time stepping schemes
with $m=1$ (left plot) and $m=2$ (right plot) as a function of the
number of solves.  Note that the GRM method leads to smaller error
using the same number of solves.   
Plots for $f$ given by (b) and (d) are similar and
are omitted.  

\begin{figure}[!htb]
	\centering
		\subfloat[ Case (c): $m=1$]{\includegraphics[width=0.50\textwidth,trim=500 40 500 40, clip]{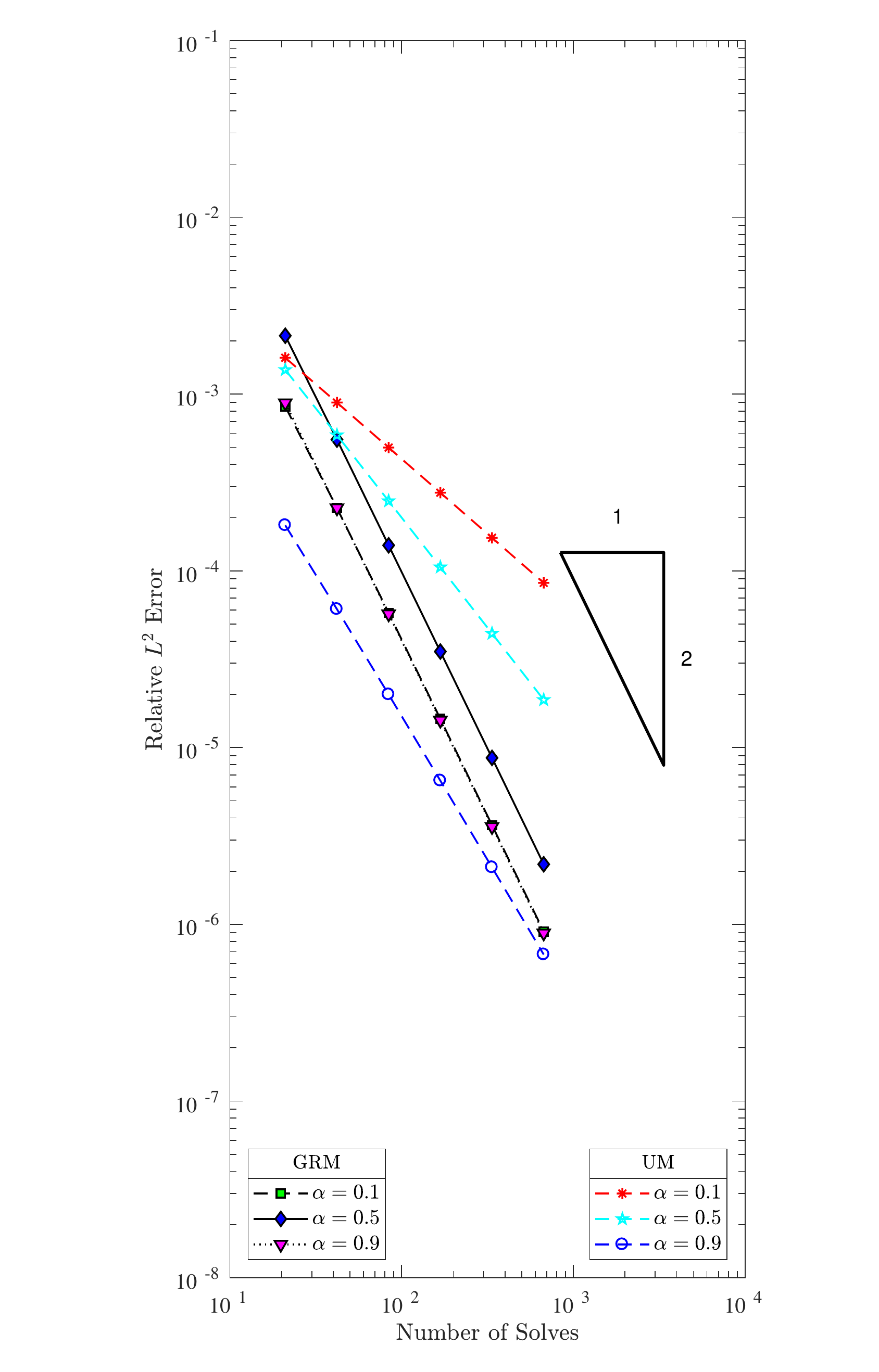}}
		\subfloat[ Case (c): $m=2$]{\includegraphics[width=0.50\textwidth,trim=500 40 500 40, clip]{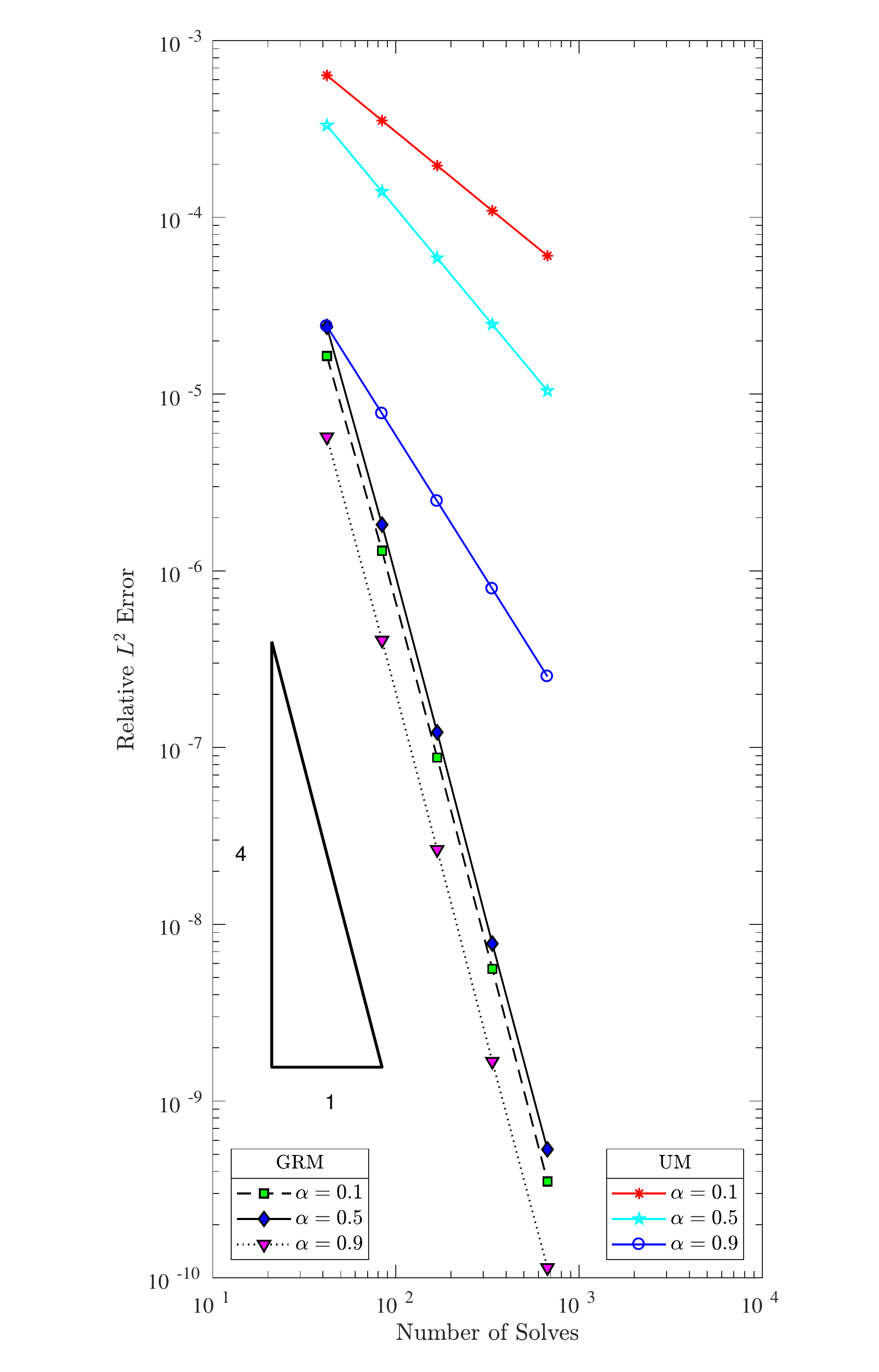}}
	\caption{Case (c):  relative $L^2$ error on geometrically
		refined mesh(GRM) and uniform mesh(UM) for $m=1$ (left) and
		$m=2$ (right).}\label{figure1d}
\end{figure}

To further demonstrate that the numerical results reflect the
theoretical results proved earlier, we report the approximate order of
convergence going from $N=n$ to $N=2n$, 
\beq 
\hbox{Approximate order}:= \log(E_{n}/E_{2n})/\log(2).\label{order}
\eeq
We used $n=8$ for both the UM runs and the GRM runs. The reason that we chose this $n$  for this computation is that the errors were getting so small in the GRM algorithm for $N\ge 32$ that, we suspect, computer round off  was effecting their significance.

\begin{table}[h!]
	\centering
	\caption{Approximate order vs. theoretical convergence rates for  $m=1$.}
	\label{table1}
	\begin{tabular}{|c|lll| lll |}
		\hline
		&\multicolumn{3}{c|}{GRM scheme.}&
		\multicolumn{3}{c|}{UM scheme.}\\
		\hline
		$\alpha$&0.1&0.5&0.9&0.1&0.5&0.9\\
		\hline
		$f$ by (a)&2.00(2)&2.00(2)&2.00(2)&1.87(2)&1.95(2)&1.99(2)\\
		$f$ by (b)&2.00(2)&2.00(2)&2.00(2)&1.34(1.35)&1.71(1.75)&1.94(2)\\
		$f$ by (c)&2.00(2)&2.00(2)&2.00(2)&0.85(0.85)&1.25(1.25)&1.63(1.65)\\
		$f$ by (d)&2.00(2)&2.00(2)&2.00(2)&0.40(0.35)&0.77(0.75)&1.16(1.15)\\
		\hline
	\end{tabular}
\end{table}

We report the approximate order of convergence computed using \eqref{order} and  compare it with the theoretical rate (in parenthesis) in 
Table \ref{table1} and Table \ref{table2}. Note that the approximate order of convergence was under the assumption that  inequality \eqref{discretepih} holds.  Table \ref{table1} and  Table \ref{table2} give the rates when $m=1$ and $m=2$, respectively for varying $\alpha$ and $f$
given above. In most cases, the computed order is in good agreement with the theoretical rate for both the refinement and uniform time stepping schemes. In contrast, the theoretical rate of the smooth problem (for $f$ given by (a)) would be $2m$ if \eqref{discretepih} held. The results in Table  \ref{table2} suggests that \eqref{discretepih} does not hold uniformly for $\beta=4-\alpha$. In all of the above examples, the error observed for the refinement scheme as a function of the number of solves was below that of the uniform time stepping scheme.  

\begin{table}[h!]
	\centering
	\caption{Approximate order vs. theoretical convergence rates, $m=2$.}
	\label{table2}
	
	\begin{tabular}{|c|ccc| lll |}
		\hline
		&\multicolumn{3}{c|}{GRM scheme.}&
		\multicolumn{3}{c|}{UM scheme.}\\
		\hline
		$\alpha$&0.1&0.5&0.9&0.1&0.5&0.9\\
		\hline
		$f$ by (a)&4.00(4)&3.98(4)&3.97(4)&2.63(4)&2.87(4)&2.96(4)\\
		$f$ by (b)&3.97(4)&3.61(4)&3.87(4)&1.35(1.35)&1.75(1.75)&2.01(2.15)\\
		$f$ by (c)&4.00(4)&3.87(4)&3.88(4)&0.85(0.85)&1.25(1.25)&1.65(1.65)\\
		$f$ by (d)&4.00(4)&3.98(4)&4.00(4)&0.42(0.35)&0.79(0.75)&1.17(1.15)\\
		\hline
	\end{tabular}
\end{table}

Note that the convergence of the GRM schemes is more robust than that of
the UM schemes. The GRM schemes always yield $2m$'th order convergence
while the convergence rate of UM schemes are related to the parameter 
$\alpha$ and the the (discrete) regularity of the initial data $v$ as
suggested by the theory.  The advantages of the refinement scheme are
especially evident for the non-smooth initial data problem.

\subsection{A spatial refinement example}

The last one dimensional example is for $f=1$ but uses a sequence of 
refined spatial grids. By  \eqref{FEM-error}, the semi-discrete error 
for an unrefined mesh is $O(h^{2s})$ for $s<1/4+\alpha$.  As the
singular behavior is at 
the endpoints of the interval, it is natural to use refinement there
to try to improve the error behavior.  We consider a mesh
resulting from a geometric refinement near 0 and 1 
similar to the geometric time stepping
refinement at 0. Specifically,  our meshs on $[0,1/2]$ are constructed
by restricting the mesh of Subsection~3.1 to [0,1/2]  as a function of
$N$,  the number of points per interval. In this
construction,  we choose   $L$ 
so that $2^{-L}\le
h^{-2}$ where $h=1/4N$ is the mesh size on [1/4,1/2].   The mesh on
[1/2,1] is obtained by reflecting the mesh on [0,1/2] about 1/2. The
number of mesh points in space is $O(h^{-1} \hbox{log} (1/h))$.

Table \ref{table3} 
reports errors using the GRM and UM
time stepping schemes applied to the  case when $\calA_h$ comes from a
sequence of refined spatial meshes as discussed above.   For brevity, we only report
results for $\alpha=0.5$. 
For each spatial mesh, we compute an accurate approximation
$u_{h,ref}$ to the
semi-discrete solution $u_h= \calA_h^{-\alpha} \pi_h f$
by using a highly refined (in time)
4'th order GRM time-stepping scheme.   We then report the
semi-discrete error norm 
$e_{semi}:=\|I_h(u) -  u_{h,ref}\|$ where $I_h$ denotes
the finite element interpolation operator on the refined spatial mesh.   The solution $u$ is
computed at the nodes
by using 800000 terms in its Fourier series expansion.
The error $e_{semi}$ is important as it gives us an idea how small we
need to make the time stepping error so that the overall error is, for
example, less than or equal to $2e_{semi}$.
In Table \ref{table3},  
$nx$ is  the number of intervals in the
spatially refined  grid,    $NS$ is the number of time steps
used to reduce  the GRM error
$E_{GRM}:=\|u_{GRM}-u_{h,ref}\|$
below $e_{semi}$
and $E_{UM}:=\|u_N-u_{h,ref}\|$ is the UM error for $N=10^5$ time steps.
It is  clear that the uniform time stepping method is
 inefficient for this problem. Indeed,  in many cases,  
the uniform time stepping fails to reduce the error 
below  $e_{semi}$ 
even when using $10^5$  time steps.

\begin{table}[h!]
	\footnotesize
	\centering
	\caption{Error and the number of steps for the local refinement is space.} %
	\label{table3}	
	\begin{tabular}{|c|c|c|c|c|c|c|c|c|}
		\hline
		\multirow{2}{*} {$N$}& \multirow{2} {*} {$nx$}&\multirow{2}{*}{$e_{semi}$}
		&\multicolumn{3}{c|}{ $m=1$}& \multicolumn{3}{c|} {$m=2$}\\
		\cline{4-9}
		&&&$E_{GRM}$&$NS$&$E_{UM}$&$E_{GRM}$&$NS$&$E_{UM}$\\
		\hline
		4&72&$6.85\times10^{-4}$&$2.24\times10^{-4}$&$92$&$2.21\times10^{-5}$
		&$7.78\times10^{-5}$&$23$&$5.47\times10^{-6}$\\
		\hline
		8&176&$1.71\times10^{-4}$&$5.62\times10^{-5}$&$232$&$2.75\times10^{-5}$
		&$7.75\times10^{-5}$&$29$&$1.44\times10^{-5}$\\
		\hline
		16&416&$4.29\times10^{-5}$&$1.41\times10^{-5}$&$560$&$2.77\times10^{-5}$
		&$6.14\times10^{-6}$&$70$&$1.47\times10^{-5}$\\
		\hline
		32&960&$1.07\times10^{-5}$&$3.52\times10^{-6}$&$1312$&$2.77\times10^{-5}$
		&$6.14\times10^{-6}$&$82$&$1.47\times10^{-5}$\\
		\hline
		64&2176&$2.68\times10^{-6}$&$8.79\times10^{-7}$&$3008$&$2.78\times10^{-5}$
		&$4.15\times10^{-7}$&$188$&$1.47\times10^{-5}$\\
		\hline
		128&4864&$6.71\times 10^{-7}$&$2.19\times10^{-7}$&$6784$&$2.77\times10^{-5}$
		&$4.15\times10^{-7}$&$212$&$1.47\times10^{-5}$\\
		\hline
	\end{tabular}
\end{table}

\subsection{Two dimensional examples: $\Omega=(0,1)^2$}
In the two dimensional case, we consider $f$ given by: 
\begin{enumerate}[(a)]
	\setcounter{enumi}{4}
	\item $f(x,y)= x(1-x)y(1-y)$   so that   $f\in \dH^s$ for $s<5/2$;
	\item $f(x,y)=\left \{ 
	\begin{matrix}
	1& \quad 0.25\le x,y\le 0.75,\\
	0& \quad \mbox{otherwise,}
	\end{matrix} \right.$ 
	\, \, so that $f\in \dH^s$ for $s< 1/2$.
\end{enumerate}

For brevity, we only report results for the case of $m=2$.  For all runs, we use a uniform mesh in space of size 
$h=1/100$ and $L=14$ in the GRM case. We report  the relative error of the time stepping solution compared with
$u_h=\mathcal{A}_h^{-\alpha} \pi_h f$. Table~\ref{Tab-UM-2d-2} gives the  errors as a function of $\alpha$ 
and $NS$, the number of time steps, for the uniform stepping approximation while those of  Table \ref{Tab-GRM-2d-2} are for the 
geometric stepping approximation.  Similar to 1-D case, the reported convergence rates are obtained by \eqref{order}
using $n:=NS=8\times 15$ with the theoretical rates in parenthesis.

\begin{table}
	\footnotesize
	\centering
	\caption{The error
		$\|U_{N}-u_{h}\|_{L^2(\Omega)}/\|u_{h}\|_{L^2(\Omega)}$
                for the fourth order UM scheme.}
	\label{Tab-UM-2d-2}
	\begin{tabular}{|c|c|cccccc|c|}
		\hline
		Ex.& $\alpha\backslash NS $  &$15$ &$2\times15$
          &$4\times15$ & $8\times15$ & $16\times15$ &$32\times 15$ &conv. rate \\
		\hline
		\multirow{5}{*}{(e)}&$0.1$& 3.10e-05 &1.24e-05 &4.98e-06 &2.00e-06 &7.87e-07 &2.87e-07 &1.31(1.35)\\
		&$0.3$& 3.62e-05 &1.26e-05 &4.39e-06 &1.54e-06 &5.35e-07 &1.76e-07 &1.52(1.55)\\
		&$0.5$& 2.19e-05 &6.60e-06 &2.00e-06 &6.13e-07 &1.87e-07 &5.49e-08 &1.72(1.75)\\
		&$0.7$& 9.39e-06 &2.47e-06 &6.51e-07 &1.73e-07 &4.62e-08 &1.20e-08 &1.92(1.95)\\
		&$0.9$& 2.06e-06 &4.71e-07 &1.08e-07 &2.50e-08 &5.83e-09 &1.44e-09 &2.12(2.15)\\
		\cline{1-9}
		\multirow{5}{*}{(f)}&$0.1$& 1.85e-02 &1.23e-02 &7.70e-03 &4.48e-03 &2.34e-03 &1.05e-03 &0.67(0.35)\\
		&$0.3$& 1.81e-02 &1.08e-02 &6.13e-03 &3.25e-03 &1.56e-03 &6.54e-04 &0.82(0.55)\\
		&$0.5$& 8.97e-03 &4.78e-03 &2.44e-03 &1.17e-03 &5.11e-04 &1.97e-04 &0.97(0.75)\\
		&$0.7$& 3.18e-03 &1.50e-03 &6.81e-04 &2.93e-04 &1.16e-04 &4.07e-05 &1.14(0.95)\\
		&$0.9$& 5.77e-04 &2.41e-04 &9.69e-05 &3.71e-05 &1.32e-05 &4.21e-06 &1.31(1.15)\\
		\hline
	\end{tabular}
\end{table}

\begin{table}
	\footnotesize
	\centering
	\caption{The error
		$\|U_{L+1}-u_{h}\|_{L^2(\Omega)}/\|u_{h}\|_{L^2(\Omega)}$
                for the fourth order GRM scheme.}
	\label{Tab-GRM-2d-2}
	\begin{tabular}{|c|c|cccccc|c|}
		\hline
		Ex.& $\alpha\backslash NS $  &$15$ &$2\times15$ &$4\times15$ & $8\times15$ & $16\times15$ &$32\times15$ & conv. rate \\
		\hline
		\multirow{5}{*}{(e)}&$0.1$& 3.67e-06 &2.79e-07 &1.86e-08 &1.20e-09 &1.01e-10 &4.33e-11 &3.91(4)\\\
		&$0.3$& 7.23e-06 &5.42e-07 &3.59e-08 &2.29e-09 &1.47e-10 &1.48e-11 &3.91(4)\\
		&$0.5$& 7.29e-06 &5.35e-07 &3.52e-08 &2.20e-09 &1.24e-10 &6.53e-11 &3.92(4)\\
		&$0.7$& 5.25e-06 &3.76e-07 &2.46e-08 &1.53e-09 &8.13e-11 &3.39e-11 &3.94(4)\\
		&$0.9$& 1.96e-06 &1.36e-07 &8.74e-09 &4.61e-10 &1.33e-10 &1.64e-10 &3.96(4)\\
		\cline{1-9}
		\multirow{5}{*}{(f)}&$0.1$& 9.86e-05 &7.99e-06 &5.46e-07 &3.50e-08 &2.20e-09 &1.45e-10 &3.87(4)\\
		&$0.3$& 1.27e-04 &1.01e-05 &6.86e-07 &4.39e-08 &2.76e-09 &1.74e-10 &3.88(4)\\
		&$0.5$& 9.10e-05 &7.16e-06 &4.83e-07 &3.08e-08 &1.93e-09 &1.34e-10 &3.89(4)\\
		&$0.7$& 4.91e-05 &3.80e-06 &2.55e-07 &1.63e-08 &1.02e-09 &7.05e-11 &3.90(4)\\
		&$0.9$& 1.39e-05 &1.06e-06 &7.08e-08 &4.49e-09 &3.06e-10 &1.63e-10 &3.90(4)\\
		\hline
	\end{tabular}
\end{table}

\section{Conclusions}

 We proposed two time-stepping methods based on Pad\'e approximation for solving a special pseudo-parabolic 
 equation  introduced by Vabishchevich for solving equations involving powers of symmetric positive elliptic operators. 
 We consider two schemes that use geometrically refined  and uniform meshes in time.  The scheme that uses geometrically refined mesh has a   
 convergence rate that does not depend on the smoothness of the solution, while the scheme involving
 uniform time-mesh depends crucially on the discrete regularity of the solution. Both, the theoretical estimates and the numerical tests show that the scheme  
 on geometrically refined meshes  is more efficient compared with the uniform time-stepping scheme, especially in the non-smooth data case.

\section*{Acknowledgments}
B.~Duan is supported by China Scholarship Council and the Fundamental Research Funds for the Central Universities of 
Central South University (2016zzts015).  The work of R.~Lazarov was supported in part by  NSF-DMS \#1620318 grant.
 
%

\bibliographystyle{abbrv}
\bibliography{references,Ray_references}

\begin{thebibliography}{10}

\bibitem{Baker1975}
G.~A. Baker.
\newblock {\em Essentials of Pad\'e approximants}.
\newblock Academic Press, 1975.

\bibitem{bankdupont}
R.~E. Bank and T.~Dupont.
\newblock An optimal order process for solving finite element equations.
\newblock {\em Math. Comp.}, 36(153):35--51, 1981.

\bibitem{bankYserentant}
R.~E. Bank and H.~Yserentant.
\newblock On the {$H\sp 1$}-stability of the {$L\sb 2$}-projection onto finite
  element spaces.
\newblock {\em Numer. Math.}, 126(2):361--381, 2014.

\bibitem{Bonito2018}
A.~Bonito, J.~P. Borthagaray, R.~H. Nochetto, E.~Ot{\'a}rola, and A.~J.
  Salgado.
\newblock Numerical methods for fractional diffusion.
\newblock {\em Computing and Visualization in Science}, Mar 2018.

\bibitem{BP-frac}
A.~Bonito and J.~Pasciak.
\newblock Numerical approximation of fractional powers of elliptic operators.
\newblock {\em Mathematics of Computation}, 84(295):2083--2110, 2015.

\bibitem{BrambleXu}
J.~H. Bramble and J.~Xu.
\newblock Some estimates for a weighted {$L\sp 2$} projection.
\newblock {\em Math. Comp.}, 56(194):463--476, 1991.

\bibitem{fujitasuzuki}
H.~Fujita and T.~Suzuki.
\newblock Evolution problems.
\newblock In {\em Handbook of numerical analysis, {V}ol.\ {II}}, Handb. Numer.
  Anal., II, pages 789--928. North-Holland, Amsterdam, 1991.

\bibitem{gilboa2008nonlocal}
G.~Gilboa and S.~Osher.
\newblock Nonlocal operators with applications to image processing.
\newblock {\em Multiscale Modeling \& Simulation}, 7(3):1005--1028, 2008.

\bibitem{Higham1997}
N.~J. Higham.
\newblock Stable iterations for the matrix square root.
\newblock {\em Numerical Algorithms}, 15(2):227--242, 1997.

\bibitem{kato1961}
T.~Kato.
\newblock Fractional powers of dissipative operators.
\newblock {\em J. Math. Soc. Japan}, 13:246--274, 1961.

\bibitem{Kenney1989}
C.~Kenney and A.~Laub.
\newblock Pad\'e error estimates for the logarithm of a matrix.
\newblock {\em Int. J. Control}, 50(3):707--730, 1989.

\bibitem{Kenney1991}
C.~Kenney and A.~J. Laub.
\newblock Rational iterative methods for the matrix sign function.
\newblock {\em SIAM J. Matrix Anal. Appl.}, 12(2):273--291, Mar. 1991.

\bibitem{KilbasSrivastavaTrujillo:2006}
A.~Kilbas, H.~Srivastava, and J.~Trujillo.
\newblock {\em Theory and {A}pplications of {F}ractional {D}ifferential
  {E}quations}.
\newblock Elsevier, Amsterdam, 2006.

\bibitem{kwasnicki2017ten}
M.~Kwa{\'s}nicki.
\newblock Ten equivalent definitions of the fractional laplace operator.
\newblock {\em Fractional Calculus and Applied Analysis}, 20(1):7--51, 2017.

\bibitem{wenyu-thesis}
W.~Lei.
\newblock {\em Numerical approximation of partial differential equations
  involving fractional differential operators}.
\newblock PhD thesis, Texas A\&M University, 2018.

\bibitem{Karniadakis2018fractional}
A.~Lischke, G.~Pang, M.~Gulian, F.~Song, C.~Glusa, X.~Zheng, Z.~Mao, W.~Cai,
  M.~M. Meerschaert, M.~Ainsworth, et~al.
\newblock What is the fractional laplacian?
\newblock {\em arXiv preprint arXiv:1801.09767}, 2018.

\bibitem{lunardi}
A.~Lunardi.
\newblock {\em Interpolation Theory}.
\newblock Edizioni dela Normale, second edition, 2007.

\bibitem{mccay1981theory}
B.~McCay and M.~Narasimhan.
\newblock Theory of nonlocal electromagnetic fluids.
\newblock {\em Archives of Mechanics}, 33(3):365--384, 1981.

\bibitem{metzler2014anomalous}
R.~Metzler, J.-H. Jeon, A.~G. Cherstvy, and E.~Barkai.
\newblock Anomalous diffusion models and their properties: non-stationarity,
  non-ergodicity, and ageing at the centenary of single particle tracking.
\newblock {\em Physical Chemistry Chemical Physics}, 16(44):24128--24164, 2014.

\bibitem{pu2010fractional}
Y.-F. Pu, J.-L. Zhou, and X.~Yuan.
\newblock Fractional differential mask: a fractional differential-based
  approach for multiscale texture enhancement.
\newblock {\em IEEE transactions on image processing}, 19(2):491--511, 2010.

\bibitem{silling2000reformulation}
S.~A. Silling.
\newblock Reformulation of elasticity theory for discontinuities and long-range
  forces.
\newblock {\em Journal of the Mechanics and Physics of Solids}, 48(1):175--209,
  2000.

\bibitem{thomeeBook}
V.~Thom{\'e}e.
\newblock {\em Galerkin finite element methods for parabolic problems}, volume
  1054.
\newblock Springer, 1984.

\bibitem{Vabishchevich2015JCP}
P.~N. Vabishchevich.
\newblock Numerically solving an equation for fractional powers of elliptic
  operators.
\newblock {\em Journal of Computational Physics}, 282:289--302, 2015.

\end{thebibliography}
\end{document}